\documentclass{amsproc}

\usepackage{amsmath}
\usepackage{amsthm}
\usepackage{amsfonts}
\usepackage{amsbsy}
\usepackage{amssymb}
\usepackage[initials]{amsrefs}
\newtheorem{theorem}{Theorem}
\newtheorem{proposition}{Proposition}
\newtheorem{corollary}{Corollary}
\newtheorem{lemma}{Lemma}

\newcommand{\R}{{\mathbb R}}

\newcommand{\Z}{{\mathbb Z}}

\newcommand{\set}[2]{ \left\{ #1 \ \left| \ #2 \right. \right\}}

\newcommand{\ang}[1]{\left< #1 \right>}

\newcommand{\GL}{\operatorname{GL}}
\newcommand{\T}{{\mathbb T}}
\newcommand{\M}{D}
\newcommand{\const}{{\mathcal C}}
\usepackage{bbm}
\newcommand{\one}{\mathbbm{1}}
\newcommand{\indicator}{\mathbbm{1}}
\newcommand{\tilefam}{\mathcal{T}_*}
\newcommand{\ts}[1]{{{#1}_\square}}

\title[Smooth, compactly-supported frames adapted to frequency tilings]{On frames of smooth, compactly-supported wave packets adapted to tilings of frequency space}
\author{Philip T. Gressman}
\thanks{The author was partially supported by NSF grant DMS-2054602.}

\begin{document}
\dedicatory{Dedicated to the memory of Guido Weiss and his generous and engaging way of sharing mathematics with students.}
\begin{abstract}
We establish a broad notion of admissible tilings of frequency space which admit associated wave packet frames with elements which are smooth and compactly supported. The framework is designed to allow for tile geometries which are minimally constrained by the need to accommodate Schwartz tails on the Fourier side and goes beyond the usual scale of geometries ranging from Gabor to wavelet-type decompositions \cite{MR2179584}. The approach builds on techniques of Hern\'{a}ndez, Labate and Weiss \cite{MR1916862} and Labate, Weiss, and Wilson \cite{MR2066831} as well as a classical result of Ingham \cite{MR1574706} characterizing the best-possible Fourier decay for functions of compact support.
\end{abstract}
\maketitle

\section{Introduction}

The article \cite{MR4302195} of the author and E. Urheim establishes stable norm decay inequalities for multilinear oscillatory integral functionals through the use of certain continuous frame decompositions of $L^2(\R)$ which behave like Gabor systems at low frequencies and so-called ``wave atoms'' at higher frequencies. Here the adjective ``continuous'' refers to decompositions which involve integrals over all translations of some countable family of window functions, where ``discrete'' would instead refer to summation over discrete lattices of translations. An issue of central concern there was the need for uniform control of frame constants independent of the frequency at which the transition occurs. At the time it was noted that the existence of analogous discrete frames with window functions that are smooth and compactly supported would likely have simplified the analysis and would also open interesting pathways to studying degenerate multilinear oscillatory integrals of a similar form. 

The construction of continuous frames whose wave packets have compact \textit{Fourier} support is essentially trivial and can be accomplished under extremely mild assumptions on the arrangement of frequency tiles (see, for example, the notion of bounded admissible partitions of unity as described in \cite{MR2179584}). Unfortunately, the switch to compact \textit{physical space} support introduces Schwartz tails on the frequency side which can, in extreme cases, have nontrivial effects. Passing from continuous to discrete frames introduces even more difficulties because the composition of analysis and synthesis operators for the frame is not translation-invariant and so a complete understanding of its properties does not follow from elementary Fourier analysis. The goal of the present article is to approach these difficulties directly for the purpose of establishing the existence of discrete frames of smooth, compactly-supported wave packets adapted to decompositions of frequency space which are as general as possible. A secondary but still important goal is to demonstrate existence of such discrete frames in a way that yields frame constants (and even choices of wave packets) which are uniform over very broad classes of decompositions so that, as was done for multilinear oscillatory integrals, it is possible to analyze functions or operators using parameter-dependent decompositions of frequency space while still proving inequalities that are independent of parameter.

Let $|\cdot|$ denote the $\ell^\infty$ norm on $\R^d$, i.e., $| (x_1,\ldots,x_d) | := \max \{|x_1|,\ldots,|x_d|\}$.  Given any matrix $D \in \GL_d(\R)$, we define a \textit{tile} $T$ to be a pair $(D,\xi) \in \GL_d (\R) \times \R^d$ and define its tile set $\ts{T} \subset \R^d$ to equal 
\[ \ts{T} := \set{ \eta \in \R^d}{ |D^{-1} (\eta-\xi)| \leq \frac{1}{2}}. \]
Modulo only the finite symmetry group generated by permutation of coordinates and coordinate-wise reflection,  the tile set $\ts{T}$ determines $T$ (see Proposition \ref{pretileprop}) so it is only a small exaggeration to informally regard tiles themselves as subsets of $\R^d$.
In particular, the notation $\one_T$ will always refer to the indicator function on $\R^d$ the tile set $\ts{T}$. 
The symbol $\mathcal T$ will denote the set of all tiles and $\dot{\mathcal T}$ indicates the set of all tiles centered at the origin, i.e., tiles of the form $(D,0)$.  
Given a tile $T$ and a real number $t >0$, the tile $t T$ is defined to be the tile obtained by dilating $T$ isotropically by a factor of $t$ in each direction while preserving its center, i.e., if $T = (D,\xi)$, then $t T := (t D,\xi)$. Similarly, $\dot T$ is the translation of $T$ which is centered at the origin, i.e., $\dot T = (D,0)$.
Given any function $\varphi$ on $\R^d$ and any tile $T = (D,\xi)$, we make the definitions
\begin{align}
\varphi_{T,k}(x) & := e^{2 \pi i \xi_0 \cdot x} \sqrt{|\det D|} \varphi(k + D^t x) & \forall x,k \in \R^d, \\
\varphi^T(\xi') & := \varphi (D^{-1} (\xi' - \xi)) & \forall \xi' \in \R^d. \label{dilation2}
\end{align}
Here $D^t$ is the usual transpose. Dilations of the first type $\varphi_{T,k}$ will primarily be applied to functions in physical space; dilations of the form $\varphi^T$ are generally reserved for functions in frequency space (so, in particular, as $\widehat{\cdot}$ will be used to denote the Fourier transform, the notation $\widehat{\eta}^T$ is to be understood as $(\widehat{\eta})^T$).

The first theorem establishes that a discrete frame with smooth, compactly-supported window $\psi$ adapted to an essentially arbitrary family of tiles $\mathcal{T}_*$ must exist whenever $\mathcal{T}_*$ is in some sense oversampling and admits a corresponding \textit{continuous} frame adapted to $\tilefam$ with compactly supported window $g$. The precise statement is as follows.
\begin{theorem}[Discrete frames from continuous frames]
Let $\mathcal{T}_*$ be any countable set of tiles in $\R^d$. \label{compactthm}
Suppose that there exists a compactly supported $g \in L^2(\R^d)$, $A,B > 0$ and $\theta \in (0,1)$ such that
\begin{equation} A \leq \sum_{T \in \mathcal{T}_*} |\widehat{g}^T \! (\xi)|^2 \one_{(1-\theta) T}(\xi) \leq \sum_{T \in \mathcal{T}_*} |\widehat{g}^T \! (\xi)|^2 \leq B \label{chyp} \end{equation}
for almost every $\xi \in \R^d$.  Then for all sufficiently small $\epsilon > 0$ (depending only on $A$, $B$, $\theta$, and $d$), there exists a compactly supported $\psi \in C^\infty(\R^d)$ such that
\begin{equation} (1-\epsilon)A ||f||^2 \leq \sum_{\substack{ T \in \mathcal{T}_* \\ k \in \Z^d}} |\ang{f,\psi_{T,k}}|^2 \leq (1 + \epsilon) B ||f||^2 \label{cframe} \end{equation}
for all $f \in L^2(\R^d)$. This $\psi$ may be taken to depend only on $g$, $A$, $B$, $\epsilon$, $\theta$, and $d$. In particular, it can be defined independently of the specific tiling $\tilefam$.
\end{theorem}
To understand the significance of the hypotheses, it is helpful to recall the consequences of \eqref{cframe}. Proposition 4.1 of Hern\'{a}ndez, Labate, and Weiss \cite{MR1916862} establishes that any $\psi$ satisfying \eqref{cframe} must also satisfy
\[ \sum_{T \in \mathcal{T}_*} |\widehat{\psi}^T\!(\xi)|^2 \leq (1 + \epsilon) B \text{ for a.e. } \xi \in \R^d \]
and an argument similar to the one presented there (testing the inequality on functions $f$ whose Fourier support is concentrated on vanishingly small balls) also establishes that
\[ \sum_{T \in \mathcal{T}_*}  |\widehat{\psi}^T\!(\xi)|^2 \geq (1-\epsilon)A  \text{ for a.e. } \xi \in \R^d. \]
Thus, modulo perturbations of the constants, the inequality \eqref{cframe} implies that \eqref{chyp} holds up to but not including the indicator function $\one_{(1-\theta) T}$.  But Theorem \ref{compactthm} cannot hold without some additional restriction of this sort; even when $\theta = 0$, for example, the Balian-Low Theorem demonstrates that the analogue of Theorem \ref{compactthm} is false (see, for example, Theorem 3.1.d in \cite{MR1350699} and the commentary immediately following that theorem; for any continuous $\psi$ of compact support, its Zak transform is continuous and therefore always has a zero).

Theorem \ref{compactthm} can be used, for example, as a means to establish the existence of wavelet frames. Here the tiles would have the form $(2^j,0)$ (for the one dimensional case) with $j$ ranging over all integers. The tile sets have infinite overlap at the origin, but as far as the hypotheses \eqref{chyp} are concerned, this infinite overlap is easily managed by simply taking a function $g$ with integral zero so that its Fourier transform vanishes at the origin. In fact, it is not especially difficult to see that \textit{any} choice of $g$ which is compactly supported, at least $C^1$, and has integral zero will satisfy \eqref{chyp} for some nonzero values of $A$ and $B$. Similar but more exotic tilings of frequency space adapted to the geometry of curves in the plane (for example, the usual decomposition of the neighborhood of a circle into boxes of size $2^{-j} \times 2^{-j/2}$ at distance roughly $2^{-j}$ to the circle) are possible as well provided that the Fourier transform of $g$ vanishes on a suitably chosen parabola. This may be of some interest in the development of axiomatic approaches to decoupling theory (see Section 7.3 of Guth's survey \cite{guth} for some interesting observations suggesting the value that such an endeavor might have).

In some sense, the question as to whether a smooth, compactly supported window function exists for a particular tiling of frequency space comes down to a technical question about the effect of Schwartz tails, much like question of the effect of the indicator function $\one_{(1-\theta) T}$ in \eqref{chyp}. This means in practice that some non-local criterion on tile geometry is necessary (non-local here meaning that one must compare tiles that are more than simply roughly adjacent or overlapping). A common approach when formulating such a criterion is to constrain the geometry of tiles to settings which are effectively spaces of homogeneous type and then to assume that the size of each tile is controlled by some slowly-growing function of its distance to the origin; see, for example, \cites{MR2885560,MR3189276,MR4044680,MR2362408}. Perhaps the most interesting example of tilings of this sort which is neither Gabor- nor wavelet-like is the case of so-called ``wave atoms,'' in which frequency space is decomposed into isotropic balls of radius $\sim R^{1/2}$ at distance $\sim R$ to the origin for $R \geq 1$ (as is done for the high frequency decomposition in \cite{MR4302195}).

What we seek to do here is to develop a notion of admissible tilings which is, in some sense, as broad as possible. This should mean, e.g., that it includes tiling geometries based on truly multiparameter tile geometries. While it must involve comparisons of tiles at some distance, the criterion used for comparison should also ideally be relatively easy to check (i.e., for any tile $T_1$, it should hopefully be necessary to compare its geometry to only finitely many other tiles $T_2$ even if there is no uniform bound on the number of comparisons as $T_1$ varies) and should not assume greater regularity than is necessary. These are the goals which motivate the second main theorem, Theorem \ref{maintheorem}, below.

To state the theorem, an auxiliary definition is in order. If $T_i := (D_i, \xi_i)$ are tiles for $i = 1,2$, let us define
\[ d(T_1, T_2) := \ln_2 \max \{ |D_1^{-1} D_2|, |D_2^{-1} D_1| \}, \]
where $|\cdot|$ here denotes the operator norm $\ell^\infty \rightarrow \ell^\infty$ and $\ln_2$ is the standard logarithm base $2$. For obvious reasons, $d$ is not literally a metric on tiles because it is independent of centers. However, it does become a genuine metric on the set of origin-centered tile sets (see Lemma \ref{metriclemma}). The choice of base $2$ for logarithms is purely for convenience, as it means that tiles which differ in size by a factor of $2$ have distance $1$. Using this notion of distance, we now state the definition of \textit{admissible tilings}, which are, roughly speaking, a decomposition of $\R^d$ into tile sets which are approximately disjoint and satisfy certain regularity properties guaranteeing that the shapes and sizes of tiles vary slowly at a local level. To make this definition, 
fix some constant $\const > 1$, some $\theta \in (0,1)$ and some nondecreasing function $K(t)$ on $[1,\infty)$ such that 
\begin{equation} K(1) \geq 1 \text{ and } \int_1^\infty \frac{\ln K(t)}{t^2} dt < \infty. \label{kcond} \end{equation}
A collection of tiles ${\mathcal T}_* \subset {\mathcal T}$ will be called an \textit{admissible tiling} of $\R^d$ with respect to $(\const,\theta,K)$ when the following criteria hold:
\begin{enumerate}
\item (Covering) Almost every $\xi \in \R^d$ belongs to the tile set $\ts{(1-\theta) T}$ for some $T \in \mathcal{T}_*$ and does not belong to $\ts{T}$ for more than $\const$ distinct tiles $T \in \mathcal{T}_*$.
\item (Regularity) There exists a function $\rho : \mathcal{T}_* \rightarrow [1,\infty)$ such that
\begin{enumerate}
\item If $T_1,T_2 \in \mathcal{T}_*$ and the tiles sets of $T_2$ and $\rho(T_1) T_1$ intersect, then $d(\dot T_1,\dot T_2) \leq \const$. In other words, the proportionally enlarged tile set of $\rho(T_1) T_1$ only intersects tile sets $\ts{(T_2)}$ having similar shape and size as measured by the tile geometry metric.
\item For each $t \geq 1$, the set  \[ \set{ \ts{\dot T}}{ \rho(T) \leq t \text{ for some } T \in \mathcal{T}_*}, \]
i.e., the set of origin-centered tile sets with the same shape as some $T \in {\mathcal T}_*$ having $\rho(T) \leq t$,
can be covered by no more than $K(t)$ metric balls of radius $\const$.  Roughly speaking, this means the diversity of shapes and sizes of tiles grows subexponentially as a function of $\rho$.
\end{enumerate}
\end{enumerate}
While the definition of admissible tiling depends on the choice of $\const$, $\theta$, and $K(t)$, it should be noted that increasing $\const$ or $K(t)$ and decreasing $\theta$ preserves the property of admissibility without the need to change the function $\rho$.

The main theorem for admissible tilings is as follows.
\begin{theorem}[Frames for admissible tilings] \label{maintheorem}
For given $\const > 1$, $\theta \in (0,1)$, dimension $d \geq 1$, nondecreasing function $K : [1,\infty) \rightarrow [1,\infty)$ satisfying \eqref{kcond} and all $\epsilon>0$ sufficiently small (depending on $\const, \theta$, $d$, and $K$) there exists a real function $\eta$ on $\R^d$ which is $C^\infty$ and compact support with the following properties:
\begin{enumerate}
\item This $\eta(x)$ is even as a function of each of each coordinate $x_i$ and invariant under permutations of the $x_i$. In particular, for any tile $T$, this means that the set of translations $\{\eta_{T,k}\}_{k \in \Z^d}$ depends only on the tile set $\ts{T}$.
\item Suppose ${\mathcal T}_*$ is an admissible tiling of $\R^d$ with respect to $(\const,\theta,K)$. Then
\begin{equation} (1-\epsilon) ||f||^2 \leq \sum_{\substack{T \in \mathcal{T}_* \\ k \in \Z^d}} \left|\ang{f,\eta_{T,k}} \right|^2 \leq (\const + \epsilon) ||f||^2 \label{themainframe} \end{equation}
for all $f \in L^2(\R^d)$. 
\end{enumerate}
\end{theorem}

For a nontrivial example of a setting in which this theorem applies, one can take boxes in the Mikhlin-H\"{o}rmander tiling of $\R^d$, i.e.,
\[ \mathcal{D} := \set { \prod_{i=1}^d [s_i 2^{j_1}, s_i 2^{j_i+1}] }{ j_1,\ldots,j_d \in \Z, \ s_1,\ldots,s_d \in \{-1,1\}}, \]
and generate an admissible tiling $\mathcal{T}_*$ by subdividing the tile $\prod_{i=1}^d [s_i 2^{j_1}, s_i 2^{j_i+1}]$ into approximately $\left[\ln_2 \left( 2 + \max_{i=1,\ldots,d} |j_i| \right) \right]^{1+\delta}$ equal parts in each direction (for a total of approximately $\left[\ln_2 \left( 2 + \max_{i=1,\ldots,d} |j_i| \right) \right]^{d(1+\delta)}$ congruent subtiles), dilating each box by a factor of $1+\delta_1$ for some small positive $\delta$, and defining $\rho(T) := \left[\ln_2 \left( 2 + \max_{i=1,\ldots,d} |j_i| \right) \right]^{1+\delta_2}$ for each tile $T \in {\mathcal T}_*$ which originally came from $\prod_{i=1}^d [s_i 2^{j_1}, s_i 2^{j_i+1}]$ (where again $\delta_2$ is a small positive number). This logarithmic subdivision is in some sense necessitated by the fact that the single window function $\eta$ promised by Theorem \ref{maintheorem} generates a frame for \textit{any} admissible tiling $\mathcal{T}_*$ with respect to $(\const,\theta,K)$ and yields uniform frame bounds over all such $\mathcal{T}_*$. The extra smallness is necessary to compensate for the inability to formulate any sort of vanishing moment conditions which properly apply to all possible $\mathcal{T}_*$. The fact that \text{only} logarithmic-type subdivision is necessary is a consequence of a precise characterization of the maximal Schwartz tail decay which can be exhibited by smooth functions of compact support; without this additional analysis, exponentially finer subdivision is all that can be tolerated.

A second sample application of Theorem \ref{maintheorem} which directly resolves the frame construction issue raised in  \cite{MR4302195} is described in the following proposition.
\begin{proposition}
 Fix any positive increasing function $r$ on $[1,\infty)$ such that \label{adjacentcorr}
\[ \liminf_{t \rightarrow \infty} \frac{\ln r(t)}{\ln \ln \ln t} > 1. \]
There is a constant $c > 0$ depending only on $r$ and the dimension $d$ such that the following is true.
Suppose that $\R^d$ is tiled by a family $\mathcal B$ of nonoverlapping closed, isotropic boxes $B$ each with side length $\ell(B) \geq 1$. If $\mathcal B$ has the property that any two adjacent boxes $B_1, B_2 \in {\mathcal B}$ (i.e., having common boundary points) also satisfy the comparability condition
\begin{equation} \max \{ \ell(B_1), \ell(B_2) \} \leq \min \{\ell(B_1), \ell(B_2) \} \left[ 1 + \frac{1}{r (\max \{ \ell(B_1), \ell(B_2) \})} \right], \label{adjacency} \end{equation}
then any collection $\tilefam$ of tiles $T$ whose tile sets are exactly $(1+2c)B$ for each $B \in \mathcal{B}$ must be an admissible tiling for some choice of $(\const,\theta,K(t))$ which depends only on this $r$ and the dimension $d$.
\end{proposition}
The tilings $\mathcal B$ associated with the geometry in \cite{MR4302195} all have the property that the boxes $B \in \mathcal B$ have integer side lengths and adjacent boxes have side lengths differing by at most $1$, so taking $r(t) := \max\{1,t-1\}$ and applying the proposition and Theorem \ref{maintheorem} provides for the existence of a single window function generating a frame for any possible geometry of this sort with constants that are independent of the particular tiling $\mathcal B$.

\begin{proof}[Proof of Proposition \ref{adjacentcorr}]
To show that Theorem \ref{maintheorem} applies uniformly, the first step is to bound the relative size of boxes $B_1$ and $B_{k+1} \in \mathcal{B}$ which are not necessarily adjacent but are connected by a sequence of at most $k$ ``hops'' from one box to an adjacent one. 
By induction on the basic condition \eqref{adjacency} and using the monotonicity of $r$, if $B_1,\ldots,B_{k+1}$ is any finite sequence of boxes in $\mathcal B$ such that adjacent boxes in the sequence share common boundary points, then
\[  \ell(B_1)  \left[ 1 + \frac{1}{r ( \ell(B_1))} \right]^{-k} \leq \ell(B_{k+1}) \leq \ell(B_1)  \left[ 1 + \frac{1}{r ( \ell(B_1))} \right]^{k}. \]
Since the function $r$ is necessarily bounded below, if the number of ``hops'' $k$ is bounded above by $d \lceil  r (\ell(B_1)) + 1 \rceil$, then the result will be that $\ell(B_{k+1})$ must be comparable to $\ell(B_1)$ with a constant $c$ that depends only on the function $r$ and the dimension. Now if $(x_1,\ldots,x_d)$ are the coordinates of any point in $B_1$, any point $(x_1 + \delta,x_2,\ldots,x_d)$ for $|\delta| \leq c \ell(B_1)$ will either belong to $B_1$ or a box adjacent to $B_1$ because the distance is too short to pass entirely out of both $B_1$ \textit{and} a box adjacent to it in the first direction. Repeating this process, i.e., moving in the first direction no more than distance $c \ell(B_1)$, at most $\lceil r (\ell(B_1)) + 1)\rceil$ separate times and then continuing in each subsequent coordinate implies that every point $(y_1,\ldots,y_d)$ with $|x-y| < c \lceil  r (\ell(B_1)) + 1 \rceil \ell(B_1)$ belongs to a box $B'$ in a chain of at most $d \lceil  r (\ell(B_1)) + 1 \rceil$ adjacencies back to $B_1$ and consequently every box $B'$ intersecting the dilated box $(1 + 2 c \lceil  r (\ell(B_1)) + 1 \rceil) B_1$ will have side length uniformly comparable to the side length of $B_1$. 

Now for each $B \in {\mathcal B}$, let $T_B$ be a tile whose tile set equals the dilation $(1 + 2c) B$, let $\tilefam$ be the collection of all such $T_B$, let $\rho(T_B) := 1 + \frac{2c}{1+2c} r(\ell(B))$, and let $\theta := 1 - \frac{1}{1+2c}$. Clearly every $\xi \in \R^d$ belongs to $(1-\theta) T$ for some $T \in \tilefam$. 
Suppose that the tile set of $\rho(T) T$ intersects the tile set of a tile $T' \in \tilefam$. Fix $B'' \in \mathcal{B}$ to be any box in $\mathcal{B}$ containing a point in the intersection $\ts{(\rho(T)T)} \cap \ts{T'}$. The tile $T'$ is itself a $(1+2c)$ dilation of some box $B' \in \mathcal B$ and $\ts{T'}$ intersects $B''$, so the side lengths of $B'$ and $B''$ must be uniformly comparable. Likewise $\rho(T)T$ is a $(1 + 2c( r(B)+1))$ dilation of some box $B \in \mathcal B$ and $\ts{(\rho(T) T)}$ intersects $B''$, so the side lengths of $B$ and $B''$ are also uniformly comparable. This implies that the side lengths of $\ts{T}$ and $\ts{T'}$ are uniformly comparable and hence the distance $d(T,T')$ is uniformly bounded above for all such $T$ and $T'$.
This establishes the first regularity condition within the definition of admissible tilings. Moreover, when combined with Proposition \ref{overlapcountprop} in Section \ref{geomsec}, the fact that the boxes in $\mathcal B$ are nonoverlapping implies a uniform upper bound on the number of distinct tiles $T$ whose tile sets can contain any fixed $\xi \in \R^d$ (since containing the same point $\xi$ implies by the argument just given that the side lengths of the corresponding boxes $B \in {\mathcal B}$ are uniformly comparable).

Now all that remains to be established is the second regularity condition concerning the growth in size of the family of tiles with $\rho(T) \leq t$.
The set of those tiles $T$ with $\rho(T) \leq t$ are exactly all tiles whose side lengths $\ell$ satisfy $1 + \frac{2c}{1+2c} r(\ell) \leq t$. By assumption on $r$, there is some $\delta > 0$ such that  $r(s) \geq (\ln \ln s)^{1 + \delta}$
for all $s$ sufficiently large. Thus if $\ell$ is larger than the threshold $s$,  it must be the case that
 \[ \ell \leq e^{e^{\left(\frac{1+2c}{2c} (t-1) \right)^{1/(1+\delta)}}}. \]
Since every isotropic tile at the origin of side length $\ell$ is at a metric distance at most $1$ from a tile with side length equal to $2^j$ for some $j$, it follows that the maximum number $K(t)$ of unit metric balls required to cover all $\ts{\dot{T}}$ with $\rho(T) \leq t$ can be taken to depend logarithmically on the upper bound for $\ell$ just established. Specifically fixing
\[ K(t) := C e^{\left(\frac{1+2c}{2c} (t-1) \right)^{1/(1+\delta)}} \]
for an appropriately-chosen constant $C$ depending on $r$ and $d$ gives a $K$ satisfying the remaining regularity condition as well as \eqref{kcond}.
\end{proof}

The structure of the rest of this paper is as follows. Section \ref{discretesec} establishes an intermediate result concerning the relationship between discrete and continuous frames \textit{a la} results of Hern\'{a}ndez, Labate and Weiss \cite{MR1916862} and Labate, Weiss, and Wilson \cite{MR2066831}. Section \ref{smoothsec} establishes the existence of certain smooth functions of compact support with optimal Fourier decay properties. Section \ref{compactsec} contains the proof of Theorem \ref{compactthm} using the results of the previous two sections. Section \ref{geomsec} establishes a number of important properties of the metric on tiles, and finally Section \ref{mainproofsec} is devoted to the proof of Theorem \ref{maintheorem}, which is formulated as an application of Theorem \ref{compactthm}.

\section{Discrete frames as approximations of continuous frames}
\label{discretesec}
The principal idea which underlies the proofs in this paper is that under appropriate conditions, one may regard discrete frames as approximations of continuous frames. This is an idea which goes back in various forms to work of Daubechies \cite{MR1066587}.  The approach followed here builds most directly on work of Hern\'{a}ndez, Labate and Weiss \cite{MR1916862} and Labate, Weiss, and Wilson \cite{MR2066831}.  The specific hypothesis \eqref{assumption2} below also has clear connections to the uncertainty principle of Gr\"{o}chenig and Malinnikova \cite{MR3010124} and in particular indicates that one should expect Theorem \ref{frametheorem} to produce only oversampled frames when the $g_j$ are sufficiently regular (meaning that the frames will not in general be Riesz bases).
\begin{theorem}[Comparing discrete and continuous frames]
For each $d \geq 1$, there exists a positive constant $c_d$ such that the following holds. \label{frametheorem}
Suppose $\{g_j\}_{j=1}^\infty$ is a sequence of functions in $L^2(\R^d)$ and $(D_j,\xi_j)$ is a sequence of tiles in $\R^d$. If there exist positive constants $0 < A \leq B$ and $\epsilon \in [0,1)$ such that the inequalities
\begin{equation} A \leq \sum_{j=1}^\infty |\widehat{g}_j(\xi)|^2 \leq B, \label{assumption1} \end{equation}
\begin{equation}  \sum_{j=1}^\infty |\widehat{g}_j(\xi)|^2 |D_j^{-1}(\xi - \xi_j)|^{d+1} \indicator_{|D_j^{-1}(\xi - \xi_j)| \geq \frac{1}{2}} \leq c_d \epsilon^2 A \label{assumption2}
\end{equation}
hold for a.e. $\xi \in \R^d$, then every $f \in L^2(\R^d)$ satisfies the inequalities
\[ (1-\epsilon ) A ||f||^2 \leq \sum_{j=1}^\infty \sum_{k \in \Z^d}   \frac{\left|\ang{f,\tau_{D^{I}_j k} g_j}\right|^2}{|\det D_j|}  \leq (1 + \epsilon ) B ||f||^2. \]
Here $D^{I}_j := (D^t_j)^{-1}$ for each $j$ and $\tau_{p} g_j(x) := g_j(x-p)$.
\end{theorem}

The essential step in proving Theorem \ref{frametheorem} is the following lemma. 
\begin{lemma} \label{onescale}
In any dimension $d \geq 1$, there is a constant $C_d$ such that for any tile $T = (D,\xi)$, any $\delta \in (0,1)$, and any $f,g \in L^2(\R^d)$,
\begin{equation}
\begin{split}
\Bigg|  \sum_{k \in \Z^d} & \frac{|\ang{f,\tau_{D^{I} k} g}|^2}{|\det D|} - \int_{\R^d} |\widehat{f}(\xi')|^2 |\widehat{g}(\xi')|^2 d \xi' \Bigg| \\
& \leq \delta \int_{\R^d} |\widehat{f}(\xi')|^2 |\widehat{g}(\xi')|^2 d \xi' + \frac{C_d}{\delta} \int_{\R^d} |\widehat{f}(\xi')|^2 |\widehat{g}(\xi')|^2 \omega^T(\xi') d \xi' 
\end{split} \label{smoothness}
\end{equation}
where $\omega(\xi') := \indicator_{|\xi'| \geq \frac{1}{2}} |\xi'|^{d+1}$.
\end{lemma}
Before proving this lemma, let us first see how the lemma very directly implies Theorem \ref{frametheorem}. 
Let $c_d := (4C_d)^{-1}$ with $C_d$ being the constant from Lemma \ref{onescale}. The lemma can be applied for each $j$ and the results summed to conclude that
\[ \left[ (1-\delta)A - \frac{ \epsilon^2 A}{4 \delta} \right]  ||f||^2 \leq \sum_{j=1}^\infty \sum_{k \in \Z^d}   \frac{\left|\ang{f,\tau_{D^I_j k} g_j}\right|^2}{|\det D_j|}  \leq \left[ (1 + \delta) B  + \frac{\epsilon^2 A}{4 \delta} \right] ||f||^2\]
for any $\delta \in (0,1)$ and any $f \in L^2(\R^d)$ because the hypotheses \eqref{assumption1} and \eqref{assumption2} of Theorem \ref{frametheorem} imply that
\[ A ||f||^2 \leq \int_{\R^d} |\widehat{f}(\xi')|^2 \sum_{j=1}^\infty |\widehat{g}_j(\xi')|^2 d \xi' \leq B ||f||^2 \]
and
\[ \int_{\R^d} |\widehat{f}(\xi)|^2 \sum_{j=1}^\infty |\widehat{g}_j(\xi')|^2 |D_j^{-1}(\xi' - \xi_j)|^{d+1} \indicator_{|D_j^{-1}(\xi' - \xi_j)| \geq \frac{1}{2}}  d \xi' \leq c_d \epsilon^2 A ||f||^2. \]
The theorem follows by fixing $\delta := \epsilon/2$, which is allowed because $\epsilon < 1$ (here one also uses the fact that $A \leq B$ to simplify the upper frame bound).

\begin{proof}[Proof of Lemma \ref{onescale}]
It suffices to assume that the right-hand side of \eqref{smoothness} is finite.
For any $f,g \in L^2(\R^d)$, Plancherel's identity dictates that
\[ \ang{f,\tau_{D^I k} g} = \int_{\R^d} f(x) \overline{g(x-D^I k)} dx = \int \widehat{f}(\xi') \overline{ e^{-2 \pi i \xi' \cdot D^I k} \widehat{g}(\xi')} d \xi', \]
and by writing $\R^d$ as a disjoint union of sets $D ( \ell + [0,1)^d)$ for $\ell \in \Z^d$ and applying the Dominated Convergence Theorem in the usual way, it follows that
\begin{equation} \sum_{\ell \in \Z^d} \widehat{f}(\xi' - D \ell) \overline{\widehat{g}(\xi' - D \ell)} \label{bracketsum} \end{equation}
is an integrable function on $D \T^d$ which converges absolutely for a.e. $\xi'$ and that
\[  \ang{f,\tau_{D^I k} g} = \int_{D \T^d} e^{2 \pi i (D^{-1} \xi') \cdot k} \sum_{\ell \in \Z^d} \widehat{f}(\xi' - D \ell) \overline{\widehat{g}(\xi' - D \ell)} d\xi'\]
for each $k \in \Z^d$. Now by Parseval's identity,
\begin{equation} \sum_{k \in \Z^d} \frac{|\ang{f,\tau_{D^I k} g}|^2}{|\det D|} =  \int_{D \T^d} \left| \sum_{k \in \Z^d} \widehat{f}(\xi' - D k) \overline{\widehat{g}(\xi' - D k)} \right|^2 d \xi'. \label{ambigfinite} \end{equation}
Note in particular that Parseval's identity implies equality \eqref{ambigfinite} when \textit{either} side of this identity is known to be finite, so it also will trivially hold when either side is infinite. Since we may assume that $f$ and $g$ are chosen in such a way that the right-hand side of \eqref{smoothness} is finite, it follows that
\begin{equation}
\begin{split}
& \sum_{k \in \Z^d}   \frac{|\ang{f,\tau_{D^I k} g}|^2}{|\det D|} - \int_{\R^d} |\widehat{f}(\xi')|^2 |\widehat{g}(\xi')|^2 d \xi' = \\
  & \int_{D \T^d} \left[ \left| \sum_{k \in \Z^d} \widehat{f}(\xi' - D k) \overline{\widehat{g}(\xi' - D k)} \right|^2 - \sum_{k \in \Z^d} \left|\widehat{f}(\xi' - D k) \overline{\widehat{g}(\xi' - D k)} \right|^2 \right] d \xi'
 \end{split} \label{subtractdiag}
 \end{equation}
 by a similar integral periodization computation to that above (note that the subtracted quantities on both sides of \eqref{subtractdiag} are known to be finite and equal because the right-hand side of \eqref{smoothness} is assumed to be finite, so \eqref{subtractdiag} continues to make sense whether or not the quantities \eqref{ambigfinite} are known to be finite). Because $f$ and $g$ both belong to $L^2$ and \eqref{bracketsum} consequently converges absolutely for a.e. $\xi'$, 
one can write
 \begin{align*} \Bigg| \sum_{k \in \Z^d} & \widehat{f}(\xi' - D k) \overline{\widehat{g}(\xi' - D k)} \Bigg|^2 \\ & = \sum_{k \in \Z^d} \sum_{\ell \in \Z^d} \widehat{f}(\xi' - D k) \overline{\widehat{g}(\xi' - D k)} \overline{\widehat{f}(\xi' - D \ell)} {\widehat{g}(\xi' - D \ell)} 
 \end{align*}
 with absolute convergence as a sum over $(k,\ell) \in \Z^d \times \Z^d$. The second term inside the integrand on the right-hand side of \eqref{subtractdiag} is exactly equal to the sum of diagonal terms $k=\ell$, and so by the triangle inequality and \eqref{subtractdiag}, the magnitude of left-hand side of \eqref{smoothness} is bounded above by
 \begin{equation} \int_{D \T^d} \sum_{k \in \Z^d} \sum_{\ell \in \Z^d\setminus \{k\}} \left| \widehat{f}(\xi' - D k) \overline{\widehat{g}(\xi' - D k)} \overline{\widehat{f}(\xi' - D \ell)} {\widehat{g}(\xi' - D \ell)} \right| d \xi'. \label{firstupper} \end{equation}
 Now suppose that $w$ is a strictly positive measurable function on $\R^d$. The precise form of $w$ will be specified momentarily, but in any case, one may write
 \begin{align*}
  \Big| \widehat{f}(\xi' - D k) & \overline{\widehat{g}(\xi' - D k)} \overline{\widehat{f}(\xi' - D \ell)} {\widehat{g}(\xi' - D \ell)} \Big| \\
  \leq &  \frac{1}{2}  \frac{|\widehat{f}(\xi' - D k) \overline{\widehat{g}(\xi' - D k)}|^2 w(D^{-1} (\xi' - Dk - \xi))}{w(D^{-1} (\xi' - \xi) -\ell)} \\ & + \frac{1}{2}  \frac{|\widehat{f}(\xi' - D \ell) \overline{\widehat{g}(\xi' - D \ell)}|^2 w(D^{-1} (\xi' - D \ell - \xi))}{w(D^{-1} (\xi' - \xi) -k)} 
  \end{align*}
  for almost every $\xi'$
 by the standard inequality between arithmetic and geometric means. Applying this to the integrand of \eqref{firstupper} and recognizing the symmetry between $k$ and $\ell$ gives that the left-hand side of \eqref{smoothness} is bounded above by
\begin{equation}
\begin{split}
 \int_{D\T^d}  & \sum_{k \in \Z^d} \sum_{\ell \in \Z^d \setminus \{k\}}  \frac{|\widehat{f}(\xi' - D k) \widehat{g}(\xi' - D k)|^2 w(D^{-1}(\xi' - D k - \xi))}{w(D^{-1}(\xi' - \xi) -  \ell)}  d \xi' \\
 = & \int_{\R^d} |\widehat{f}(\xi')|^2 |\widehat{g}(\xi')|^2 w(D^{-1} (\xi' - \xi)) \sum_{\ell' \in \Z^d \setminus \{0\}} \frac{1}{w(D^{-1}(\xi' - \xi) - \ell')} d \xi', 
\end{split} \label{almostdone}
\end{equation}
where the equality in \eqref{almostdone} is a consequence of writing $\ell = \ell' + k$ and unperiodizing the integral.
Now take $w$ to have the form
\[ w(\xi) = a \indicator_{|\xi| < \frac{1}{2}} + |\xi|^{d+1} \indicator_{|\xi| \geq \frac{1}{2}} \]
for a positive value of $a$ which will depend on $\delta$. The key computation required to complete the proof is an upper bound for the sum
\[ \sum_{\ell' \in \Z^d \setminus \{0\}} \frac{1}{ w( \xi' - \ell')} \]
in terms of $\xi' \in \R^d$. Now
\[ \frac{1}{w(\xi)} = a^{-1}  \indicator_{|\xi| < \frac{1}{2}} + |\xi|^{-d-1} \indicator_{|\xi| \geq \frac{1}{2}}, \]
so
\begin{align*}
\sum_{\ell' \in \Z^d \setminus \{0\}} \frac{1}{w(\xi' - \ell')} & = \sum_{\ell' \in \Z^d \setminus \{0\}} \left[ a^{-1}  \indicator_{|\xi'-\ell'| < \frac{1}{2}} + |\xi'-\ell'|^{-d-1} \indicator_{|\xi'-\ell'| \geq \frac{1}{2}} \right] \\
& \leq a^{-1} \indicator_{|\xi'| > \frac{1}{2}} + \sum_{\ell' \in \Z^d} \frac{\indicator_{|\xi' - \ell'| \geq \frac{1}{2}}}{|\xi' - \ell'|^{d+1}}
\end{align*}
because $\indicator_{|\xi'-\ell'| < \frac{1}{2}}$ vanishes for all $\ell' \neq 0$ when $|\xi'| \leq \frac{1}{2}$ and otherwise is equal to $1$ for at most one $\ell'$ for each fixed $\xi'$. Now
\begin{align*}
 \sum_{\ell' \in \Z^d} \frac{\indicator_{|\xi' - \ell'| \geq \frac{1}{2}}}{|\xi' - \ell'|^{d+1}} & = \sum_{\ell' \in \Z^d}  \indicator_{|\xi' - \ell'| \geq \frac{1}{2}} \int_{\frac{1}{2}}^\infty \indicator_{|\xi'-\ell'| < s} \frac{d+1}{s^{d+2}} ds \\
\\ & = (d+1) \int_{\frac{1}{2}}^\infty \# \set{ \ell' \in \Z^d}{ \frac{1}{2} \leq |\xi'-\ell'| < s} \frac{ds}{s^{d+2}}. 
\end{align*}
Recalling that $|\cdot|$ denotes the $\ell^\infty$ norm on $\R^d$, the number of lattice points $\ell'$ inside the box $|\xi' - \ell'| < s$ never exceeds $(2s+1)^d$, which is itself less than $(4s)^d$ because $s > \frac{1}{2}$. Upon applying this inequality to the integrand immediately above, the sum in question can never exceed $2^{2d+1}(d+1)$ and therefore
\[ \sum_{\ell' \in \Z^d \setminus \{0\}} \frac{1}{w(\xi' - \ell')}  \leq a^{-1} \indicator_{|\xi'| > \frac{1}{2}} + 2^{2d+1}(d+1) \]
for every $\xi' \in \R^d$. 
Multiplying both sides by $w(\xi')$ gives that
\begin{align*}
 w(\xi') & \sum_{\ell' \in \Z^d \setminus \{0\}} \frac{1}{w(\xi' - \ell')} \\ & \leq 2^{2d+1}(d+1) a \indicator_{|\xi'| < \frac{1}{2}} + (2^{2d+1}(d+1) + a^{-1}) |\xi|^{d+1} \indicator_{|\xi'| \geq \frac{1}{2}} \\
 & \leq 2^{2d+1} (d+1) a + (2^{2d+1}(d+1) + a^{-1}) |\xi'|^{d+1} \indicator_{|\xi'| \geq \frac{1}{2}}.
 \end{align*}
 For any $\delta \in (0,1]$, fix $a := \delta 2^{-2d-1}/(d+1)$. In this case, $2^{2d+1}(d+1) \leq a^{-1}$, so
 \begin{equation} w(\xi') \sum_{\ell' \in \Z^d \setminus \{0\}} \frac{1}{w(\xi' - \ell')} \leq \delta + 2^{2d+2} (d+1) \delta^{-1} |\xi'|^{d+1} \indicator_{|\xi'| \geq \frac{1}{2}}. \label{whichineq} \end{equation}
 Applying the inequality \eqref{whichineq} with $\xi'$ replaced by  $D^{-1}(\xi' - \xi)$ to the integrand inside the right-hand side of \eqref{almostdone} completes the lemma when one takes $C_d := 2^{2d+2}(d+1)$.
\end{proof}

\section{Smooth compactly-supported functions with subexponential Fourier decay}
\label{smoothsec}
Before dealing directly with the proof of Theorem \ref{compactthm}, we first record in this section a number of basic results concerning the Fourier decay of functions which are $C^\infty$ and compactly supported. While it is widely known that such functions $\varphi$ must have Fourier transforms $\widehat{\varphi}(\xi)$ which decay faster than $C_N |\xi|^{-N}$ for any positive $N$, it is perhaps much less widely known that there exists a sharp characterization,  due to Ingham \cite{MR1574706}, of the possible Fourier decay rates exhibited by such functions $\varphi$. Proposition \ref{smoothfn}, included here for completeness and its intrinsic interest, records one direction of Ingham's characterization following essentially the same argument that appears there. For the converse direction, Ingham shows that when the integral in \eqref{wcond} is infinite,  any function $\varphi$ satisfying \eqref{ptwisebnd} must be quasi-analytic and therefore not compactly supported.
\begin{proposition}[Rapid Fourier decay of compactly-supported functions]
Let $W : \R \rightarrow [1,\infty)$ be an even function nondecreasing on $[0,\infty)$ such that $W(t) = 1$ for all $t \in [0,t_*]$ for some positive $t_*$, $t^{-N} W(t) \rightarrow \infty$ as $t \rightarrow \infty$ for all $N > 0$, and
\begin{equation} C_W := \frac{1}{t_*} + \int_0^\infty \frac{\ln W(t)}{t^2} dt < \infty. \label{wcond} \end{equation}
There exists an even, radial decreasing $C^\infty$ function $\varphi$ on $\R$ supported on the interval $[-\frac{\sqrt{e}}{\pi} C_W, \frac{\sqrt{e}}{\pi}C_W]$ such that $\varphi$ and $\widehat \varphi$ are both real and nonnegative, $\int \varphi = 1$, and \label{smoothfn}
\begin{equation} |\widehat{\varphi}(\xi)| \leq \frac{1}{W(\xi)}  \text{ for all } \xi \in \R. \label{ptwisebnd} \end{equation}
\end{proposition}
\begin{proof}
For each nonnegative integer $j$, let $t_j$ be the infimum of all $t \in[t_*,\infty)$ such that $\ln W(t) \geq j$. By definition of $t_j$, if $t_j < t < t_{j+1}$ for some $j \geq 0$, then $j \leq \ln W(t) < j+1$. It follows that for each positive $N$,
\[ \sum_{j=0}^N j \left(\frac{1}{t_j} - \frac{1}{t_{j+1}} \right) \leq \sum_{j=0}^N \int_{t_j}^{t_{j+1}} \frac{\ln W(t)}{t^2} dt \leq \int_0^{t_{N+1}} \frac{\ln W(t)}{t^2} dt  \]
and
\[ \frac{N+1}{t_{N+1}} \leq \int_{t_{N+1}}^\infty \frac{\ln W(t)}{t^2} dt. \]
Summing the two inequalities, observing that $t_0 = t_*$, and letting $N \rightarrow \infty$ gives
\[ \sum_{j=0}^{\infty} \frac{1}{t_j} \leq \frac{1}{t_*} + \int_0^\infty \frac{\ln W(t)}{t^2} dt = C_W. \] 

For some $c > 0$ to be specified shortly, let $\eta_j(x) := c^{-1} t_j \indicator_{[-c t_j^{-1} /2,c t_j^{-1} /2]}(x)$ for each $j \geq 0$. Let $\varphi_M$ equal the repeated convolution $\eta_0 * \eta_0 * \eta_1 * \eta_1 * \cdots * \eta_M * \eta_M$, and note that each $\varphi_M$ is an even, nonnegative function which is radially decreasing, has integral $1$, and is supported on the interval
\begin{equation*} \left[ - \sum_{j=0}^M \frac{c}{t_j}, \sum_{j=1}^M \frac{c}{t_j} \right] \subset \left[- c C_W, c C_W \right].  \end{equation*}
The Fourier transform of $\varphi_M$ is a product:
\begin{equation} \widehat{\varphi_M}(\xi) = \prod_{j=0}^M \left[ \frac{\sin c t_j^{-1} \pi \xi}{  c t_j^{-1} \pi \xi} \right]^2. \label{partialproducts} \end{equation} 
Since $|\sin t |/ |t| \leq 1$, these products are nonnegative and nonincreasing as a function of $M$ for each fixed $\xi$, so there is pointwise convergence of $\widehat{\varphi_M}(\xi)$ as $M \rightarrow \infty$ for each $\xi$. Because $(\sin \pi  \xi)^2/(\pi \xi)^2 \in L^1(\R)$, the partial products also converge in $L^1(\R)$ by dominated convergence. Then by Hausdorff-Young, the functions $\varphi_M$ necessarily converge uniformly as $M \rightarrow \infty$. Let $\varphi$ be the uniform limit of the functions $\varphi_M$. Uniform convergence implies that $\varphi$ is even, nonnegative, radially-decreasing, and supported on the interval $\left[- c C_W, c C_W \right]$. Moreover, the integral of $\varphi$ must equal $1$ because uniform convergence of these compactly-supported functions implies convergence in $L^1(\R)$. Finally,
\[ \int_{\R} e^{-2 \pi i x \xi} \varphi(x) dx = \lim_{M \rightarrow \infty} \int_{\R} e^{-2 \pi i x \xi} \varphi_M(x) dx = \prod_{j=0}^\infty \left| \frac{\sin c t_j^{-1} \pi \xi}{c t_j^{-1} \pi \xi} \right|^2 \]
for each $\xi \in \R$ by virtue of convergence of the sequence $\{\varphi_M\}$ in $L^1$ as well.
Now fix $c := e^{1/2} \pi^{-1}$. For any given $\xi \in \R$, if $|\xi| \geq t_{j_0}$, then one  must have
\[ \left| \frac{\sin c t_{j_0}^{-1} \pi \xi}{ c t_{j_0}^{-1} \pi \xi}\right|^2 \leq e^{-1}. \]
Taking a product over $j \leq j_0$ gives that
\[ |\widehat{\varphi}(\xi)| = \prod_{j=0}^\infty \left| \frac{\sin c t_j^{-1} \pi \xi}{c t_j^{-1} \pi \xi}\right|^2 \leq \prod_{j = 0}^{j_0} \left| \frac{\sin c t_j^{-1} \pi \xi}{c t_j^{-1} \pi \xi}\right|^2 \leq  \prod_{j = 0}^{j_0}  e^{-1} = e^{-j_0 - 1} \]
for all $|\xi| \geq t_{j_0}$. For each fixed $\xi \in \R \setminus [-t_*,t_*]$, if $j_0$ is the maximal index such that $|\xi| \geq t_{j_0}$, then $|\xi| < t_{j_0+1}$, meaning that $\ln W(\xi) < j_0 + 1$ (otherwise the definition of $t_{j_0+1}$ would require that $t_{j_0+1} \leq \ln W(|\xi|)$ and thus $t_{j_0+1} \leq |\xi|$). Therefore $|\xi| \geq t_*$ implies that $|\widehat{\varphi}(\xi)| \leq (W(\xi))^{-1}$. The same inequality holds when $|\xi| \leq t_*$ as well simply because the $L^1$-norm of $\varphi$ is one, so $|\widehat{\varphi}(\xi)| \leq 1 = 1 / W(\xi)$. Finally, note that $\varphi$ is $C^\infty$ because its Fourier transform decays faster than $|\xi|^{-N}$ for any $N > 0$.
\end{proof}

The utility of Proposition \ref{smoothfn} comes primarily through the following corollary, which establishes that one can always find nice $C^\infty$ functions of compact support $\eta$ whose Fourier transform is close to the indicator function $\indicator_{[-(1-\theta)/2,(1-\theta)/2]}$ with subexponential Schwartz tails.
\begin{corollary}[Compactly-supported approximations of tile set multipliers]
Let $W$ be as in the previous proposition. There are constants $C$ and $C'$ depending only on $W$ such that for any $\epsilon, \theta \in (0,1)$ and any dimension $d \geq 1$, there is a a real, $C^\infty$ function $\eta$ supported in $[-C \theta^{-1} \epsilon^{-1}, C \theta^{-1} \epsilon^{-1}]^d$ such that $\widehat{\eta}$ is real, nonnegative, bounded above by $1$, even in each coordinate, invariant under permutation of coordinates, and satisfies \label{almostbl}
\begin{equation} \left( 1 - \frac{C'}{W(\epsilon^{-1})} \right)^d \leq \widehat{\eta}(\xi) \text{ for all } |\xi| \leq \frac{1-\theta}{2} \label{almostcharfn} \end{equation}
and
\begin{equation} \widehat{\eta}(\xi) \leq \frac{C'}{W(\epsilon^{-1} \xi )} \text{ for all } |\xi| \geq \frac{1}{2}. \label{schwartztail} \end{equation}
\end{corollary}
\begin{proof}
If $d > 1$, it suffices to fix $\eta(\xi)$ as the product $\eta(\xi_1)\cdots \eta(\xi_d)$ for the $\eta$ constructed in the $d=1$ case. It therefore suffices to assume that $d=1$.

For $\varphi$ as given by Proposition \ref{smoothfn}, let $\kappa := ||\varphi||_{L^2}^{-2}$ and consider the function
\begin{align*}
\widehat{\eta}(\xi) := \int_{|\xi'| \leq \frac{1}{2}-\frac{\theta}{6}} \left(\frac{\epsilon \theta}{3}\right)^{-1} \kappa |\widehat{\varphi} (3 \theta^{-1} \epsilon^{-1} (\xi - \xi'))|^2 d \xi'.
\end{align*}
This $\widehat \eta$ is exactly a convolution with a dilated version of $\kappa |\widehat{\varphi}|^2$ and the indicator function of the interval $|\xi'| \leq \frac{1}{2} - \frac{\theta}{6}$. As such, $\eta$ will necessarily be a pointwise product of the inverse Fourier transform of $\indicator_{|\xi'| \leq 1/2-\theta/6}$, which is a sinc function, and a dilated version of $\varphi$ convolved with itself, which will be a $C^\infty$ function supported in the interval $[-6 \theta^{-1} \epsilon^{-1} \frac{\sqrt{e}}{\pi}C_W,6 \theta^{-1} \epsilon^{-1} \frac{\sqrt{e}}{\pi}C_W]$. Thus, $\eta$ is real, even, $C^\infty$, and supported on this same interval. The Fourier transform $\widehat{\eta}$ is everywhere nonnegative and bounded above by $1$ by virtue of the choice of $\kappa$ and the fact that the indicator function $\indicator_{|\xi'| \leq 1/2 - \theta/6}$ is simply bounded above by $1$.

Now when $|\xi| \geq 1/2$ and $|\xi'| \leq \frac{1}{2} - \frac{\theta}{6}$, it follows that
\[ |\xi - \xi'| \geq |\xi| - |\xi'| \geq |\xi| - \left(\frac{1}{2} - \frac{\theta}{6} \right) \geq |\xi| - |\xi| \left( 1 - \frac{\theta}{3} \right) \geq \frac{|\xi| \theta}{3}. \]
A change of variables gives that
\begin{align*} \widehat{\eta}(\xi) & \leq \int_{|\xi - \xi'| \geq \frac{\theta |\xi|}{3}}  \left(\frac{\epsilon \theta}{3}\right)^{-1}\kappa \left| \widehat{\varphi} (3 \theta^{-1} \epsilon^{-1} (\xi - \xi')) \right|^2 d \xi' \\ &  \leq  \int_{|\xi'| \geq \epsilon^{-1} |\xi|}  \kappa \left| \widehat{\varphi} (\xi') \right|^2 d \xi' \leq \frac{\int \kappa \widehat{\varphi}(\xi') d \xi}{W(\epsilon^{-1} \xi)}  = \frac{C'}{ W(\epsilon^{-1} \xi)} \end{align*}
for some $C'$ depending only on $W$.
Similarly, since
\begin{equation} 1 - \widehat{\eta}(\xi) = \int_{|\xi'| > \frac{1}{2}-\frac{\theta}{6}} \left(\frac{\epsilon \theta}{3}\right)^{-1} \kappa |\widehat{\varphi} (3 \theta^{-1} \epsilon^{-1} (\xi - \xi'))|^2 d \xi'\label{reverseint} \end{equation}
(because the integral of $\left(\frac{\epsilon \theta}{3}\right)^{-1} \kappa |\widehat{\varphi} (3 \theta^{-1} \epsilon^{-1} (\xi - \xi'))|^2$ with respect to $\xi'$ is always $1$),
it follows by the same reasoning that when $|\xi| < \frac{1}{2} ( 1 - \theta)$ and $|\xi'| > \frac{1}{2} - \frac{\theta}{6}$,  $|\xi' - \xi| \geq \frac{\theta}{3}$ on the support of the integral on the right-hand side of \eqref{reverseint} and consequently
\[ 1 - \widehat{\eta}(\xi)  \leq \int_{|\xi'| \geq \epsilon^{-1}}  \kappa \left| \widehat{\varphi_0} (\xi') \right|^2 d \xi' \leq \frac{C'}{W(\epsilon^{-1})}. \]
This completes the proof of the corollary.
\end{proof}

\section{Proof of Theorem \ref{compactthm}}
\label{compactsec}
\begin{proof}[Proof of Theorem \ref{compactthm}]
The proof of Theorem \ref{compactthm} follows rather immediately from Theorem \ref{frametheorem} and the constructions of the previous section.
Recall the definition \eqref{dilation2} of dilations $\widehat{\psi}^T$, and let $\psi$ be the convolution of $g$ with $\eta$ as given by Corollary \ref{almostbl} for some $W$ and $\epsilon$ to be chosen later (with $\theta$ taken to coincide with its value from the statement of Theorem \ref{compactthm}). With this identification, $\widehat{\psi}^T(\xi) = \widehat{g}^T(\xi) \widehat{\eta}^T(\xi)$ for each $\xi \in \R^d$ and each $T \in \tilefam$. Moreover $\psi$ will necessarily be $C^\infty$ and compactly supported because it is the convolution of the compactly supported functions $g$ and $\eta$ and $\eta$ is smooth. First observe that
\[ \sum_{T \in \tilefam} |\widehat{\psi}^T(\xi)|^2 = \sum_{T \in \tilefam} |\widehat{g}^T(\xi)|^2 |\widehat{\eta}^{T} (\xi)|^2 \leq  \sum_{j=1}^\infty |\widehat{g}^T(\xi)|^2 \leq B\]
for almost every $\xi \in \R^d$ by virtue of the hypothesis \eqref{chyp}. Likewise by \eqref{almostcharfn},
\begin{align*} \sum_{T \in \tilefam} |\widehat{\psi}^T(\xi)|^2 & = \sum_{T \in \tilefam} |\widehat{g}^T(\xi)|^2 |\widehat{\eta}^{T} (\xi)|^2 \\ & \geq \left( 1 - \frac{C'}{W(\epsilon^{-1})} \right)^{2d} \sum_{T \in \tilefam} |\widehat{g}(\xi)|^2 \one_{(1-\theta)T}(\xi) \geq \left( 1 - \frac{C'}{W(\epsilon^{-1})} \right)^{2d} A \end{align*}
for almost every $\xi \in \R^d$. Thus
\begin{equation} \left( 1 - \frac{C'}{W(\epsilon^{-1})} \right)^{2d} A \leq \sum_{T \in \tilefam} |\widehat{\psi}^T(\xi)|^2 \leq B \label{framebounds} \end{equation}
for almost every $\xi \in \R^d$. This is exactly the assumption \eqref{assumption1} of Theorem \ref{frametheorem} with the $\widehat{g}_j$'s taken to be equal to $\widehat{\psi}^T$ for each $T \in \tilefam$. Since $W(\epsilon^{-1})$ tends to infinity as $\epsilon \rightarrow 0^+$, the constant on the left-hand side of \eqref{framebounds} can be made as close to $A$ as desired for suitably small $\epsilon$ depending only on $W$, $\theta$, $d$, and $A$.

Next consider the analogue of \eqref{assumption2}. If $T$ is the tile $(D,\xi)$, then 
\begin{align*}
\sup_{\xi'} & |D^{-1} (\xi' - \xi)|^{d+1} \indicator_{ |D^{-1} (\xi' - \xi)| \geq \frac{1}{2}} |\widehat{\eta}^{T} (\xi')|^2   = \sup_{|\xi'| \geq \frac{1}{2}} |\xi'|^{d+1}  |\widehat{\eta}(\xi')|^2 \\ & \leq \sup_{t \geq \frac{1}{2}} \frac{(C')^2 t^{d+1}}{(W (\epsilon^{-1} t ))^2} =  \epsilon^{d+1} \sup_{t \geq \frac{1}{2}} \frac{(C')^2 (\epsilon^{-1} t)^{d+1}}{(W (\epsilon^{-1} t ))^2} \leq \epsilon^{d+1} \sup_{t \geq 0} \frac{(C')^2 t^{d+1}}{(W (t ))^2}, \end{align*}
which means that when $\omega(\xi') := |\xi'|^{d+1} \indicator_{|\xi'| \geq \frac{1}{2}}$, 
\begin{align*} \sum_{T \in \tilefam} |\widehat{\psi}^T(\xi')|^2 \omega^T(\xi) & \leq \left( \sum_{T \in \tilefam} |\widehat{g}^T(\xi')|^2 \right) \epsilon^{d+1} \sup_{t \geq 0} \frac{(C')^2 t^{d+1}}{(W (t ))^2}  \\ & \leq B (C')^2 \epsilon^{d+1} \sup_{t \geq 0} \frac{ t^{d+1}}{(W (t ))^2}. \end{align*}
In particular, for any fixed $W$, this right-hand side can be made as small as desired by choosing a sufficiently small $\epsilon$ (depending only on $A$, $B$ and $d$, and $\theta$), if, for example, $W(t) := \max\{1,e^{\sqrt{|t|}-1}\}$. Thus, Theorem \ref{compactthm} necessarily follows from Theorem \ref{frametheorem}.
\end{proof}

\section{Geometry of tile space and admissible tilings}
\label{geomsec}
This section lays the groundwork for the proof of Theorem \ref{maintheorem} by establishing some basic facts about tiles in $\R^d$. The key result of this section is Proposition \ref{overlapcount}, which establishes that, under the assumption that a set of tiles $\mathcal{T}_*$ is an admissible tiling, almost every point $\xi \in \R^d$ has the property that the number of tiles $T \in \mathcal{T}_*$ such that the tile set $\ts{t T}$ contains $\xi$ grows at most subexponentially as a function of $t$. This is the key observation which is necessary to establish that there exists a \textit{continuous} frame with $C^\infty$ window of compact support which is adapted to $\mathcal{T}_*$.

\begin{proposition}[Tile sets and symmetries]
\label{pretileprop} Suppose that $\M_1, \M_2 \in \GL_d(\R)$. The following are equivalent. 
\begin{enumerate}
\item $|\M_1^{-1} \M_2| = |\M_2^{-1} \M_1| = 1$ (where, as before, $|\cdot|$ is the $\ell^\infty$ operator norm).
\item $\M_2 = \M_1 P$ for some matrix $P$ for which there exists a permutation $\sigma$ of $\{1,\ldots,d\}$ such that $|P_{ij}| = 1$ when $i = \sigma(j)$ and $P_{ij} = 0$ otherwise.
\item $\M_1$ and $\M_2$ generate identical tile sets at the origin, i.e., $\ts{(T_1)} = \ts{(T_2)}$ when $T_i := (\M_i,0)$ for $i=1,2$.
\end{enumerate}
\end{proposition}
\begin{proof}
{[$1 \Rightarrow 2$]} Let $P := \M_1^{-1} \M_2$. By hypothesis, $|P| \leq 1$. Moreover, $P^{-1} = \M_2^{-1} \M_1$, so $|P^{-1}| \leq 1$ as well. However, since $I = P P^{-1}$ and $1 \leq |P| |P^{-1}|$, it follows that $|P| = |P^{-1}| = 1$. But as $x = P^{-1} P x$ for each $x \in \R^d$, $|x| \leq |P^{-1}| |P x| \leq |P x|$, and consequently $|x| = |Px|$ for each $x$.  Because the norm $|\cdot|$ is the $\ell^\infty$ norm, $|P| = \max_{i} \sum_{j} |P_{ij}|$. For each $j$, let $e_j$ be the $j$-th standard basis vector. Since $1 = |e_j| = |P e_j|$, it follows that there must be some index $\sigma(j) \in \{1,\ldots,d\}$ for which $|P_{\sigma(j) j}| = 1$. Since the norm of $P$ is no more than $1$, this implies that $P_{\sigma(j)j'} = 0$ when $j' \neq j$. This forces $\sigma(j) \neq \sigma(j')$ when $j \neq j'$, and consequently $\sigma$ must be a permutation of $\{1,\ldots,d\}$.

{[$2 \Rightarrow 3$]} The identity $\M_2 = \M_1 P$ implies that $\M_2^{-1} = P^{-1} \M_1^{-1}$. But $P^{-1}$ must also clearly have the same form as $P$, namely some permutation of its rows converts it to a diagonal matrix with diagonal entries $\pm 1$ only. As a consequence, $|P^{-1} x| = |x|$ for all $x \in \R^d$, and thus $|\M_2^{-1} \eta| = |\M_1^{-1} \eta|$ for all $\eta \in \R^{d}$, which in particular forces the tiles at the origin generated by $\M_1$ and $\M_2$ to be identical.

{[$3 \Rightarrow 1$]} Invertibility of $\M_1$ implies that a vector $\eta$ has $|\M_1^{-1} \eta| \leq \frac{1}{2}$ (which is exactly the definition of the tile set of the tile $(\M_1,0)$)  if and only if $\eta = \M_1 \xi$ for some $\xi$ with $|\xi| \leq \frac{1}{2}$, and equality of tile sets implies that every $\xi$ with $|\xi|\leq \frac{1}{2}$ admits some $\xi'$ also with $|\xi'| \leq \frac{1}{2}$ such that $\M_1 \xi = \M_2 \xi'$. This implies $\xi' = \M_2^{-1} \M_1 \xi$, and consequently $|\M_2^{-1} \M_1 \xi| \leq \frac{1}{2}$ for all $\xi$ with $|\xi| \leq \frac{1}{2}$. This gives directly by the definition of the operator norm that $|\M_2^{-1} \M_1| \leq 1$. By symmetry, $|\M_1^{-1} \M_2| \leq 1$ as well. But since $I = \M_2^{-1} \M_1 \M_1^{-1} \M_2$, this means that $1 \leq |\M_2^{-1} \M_1| |\M_1^{-1} \M_2|$, giving that $1 = |\M_1^{-1} \M_2| = |\M_2^{-1} \M_1|$.
\end{proof}

\begin{lemma}[Metric on tile sets]
For any \label{metriclemma}
$\M_1,\M_2 \in \GL_{d}(\R)$, let \[ d(\M_1,\M_2) := \ln_2 \max \{ |\M_1^{-1} \M_2|, |\M_2^{-1} \M_1|\}. \] Then $d(\M_1,\M_2)$ is nonnegative, symmetric, and satisfies the triangle inequality. Moreover, every $x \in \R^d$ satisfies
\begin{equation} 2^{-d(\M_1,\M_2)} | \M_1^{-1} x| \leq |\M_2^{-1} x| \leq 2^{d(\M_1,\M_2)} |\M_1^{-1} x|. \label{comparability} \end{equation}
Given any tiles $T_1,T_2 \in \dot {\mathcal{T}}$ with $T_i := (\M_i,0)$ for $i=1,2$, let 
\begin{equation} d(T_1,T_2) := d(\M_1,\M_2). \label{metricdef} \end{equation}
If $T_1'$ and $T_2'$ are any tiles in $\dot {\mathcal T}$ such that $\ts{(T_i)} = \ts{(T_i')}$ for $i=1,2$, then $d(T_1,T_2) = d(T_1',T_2')$ in particular $d$ defines a metric on the space of tile sets at the origin.
\end{lemma}
\begin{proof}
First we show the properties of $d$ as defined on pairs of invertible matrices $(\M_1,\M_2)$.
Nonnegativity of $d$ holds because $I = \M_1^{-1} \M_2 \M_2^{-1} \M_1$ implies $1 \leq |\M_1^{-1} \M_2| |\M_2^{-1} \M_1|$. Since the product is at least one, the maximum must be at least one as well, so $d(\M_1,\M_2) \geq 0$. In a similarly immediate way, taking logarithms of the inequalities $|\M_1^{-1} \M_3| \leq |\M_1^{-1} \M_2| |\M_2^{-1} \M_3|$ and $|\M_3^{-1} \M_1| \leq |\M_3^{-1} \M_2| |\M_2^{-1} \M_1|$ implies the triangle inequality. Finally, for any $\M_1,\M_2 \in \GL_d(\R)$, it must be the case that
\[ |\M_2^{-1} x| = |\M_2^{-1} \M_1 \M_1^{-1} x| \leq |\M_2^{-1} \M_1| |\M_1^{-1} x| \leq 2^{d(\M_1,\M_2)} |\M_1^{-1} x| \]
and
\[ |\M_1^{-1} x| = |\M_1^{-1} \M_2 \M_2^{-1} x| \leq |\M_1^{-1} \M_2| |\M_2^{-1} x| \leq 2^{d(\M_1,\M_2)} |\M_2^{-1} x| \] 
which together imply \eqref{comparability}.

Now suppose $T_1,T_2 \in \dot{\mathcal{T}}$ are tiles at the origin generated by matrices $\M_1$ and $\M_2$, respectively. If $\M_1'$ and $\M_2'$ are any other matrices generating tiles $T_1'$ and $T_2'$ with the property that $\ts{(T_i)} = \ts{(T_i')}$ for $i=1,2$, then by Proposition \ref{pretileprop}, $d(\M_1,\M_1') = 0 = d(\M_2,\M_2')$, and by the triangle inequality it follows that $d(\M_1,\M_2) \leq d(\M_1,\M_1') + d(\M_1',\M_2') + d(\M_2',\M_2) = d(\M_1',\M_2')$. By symmetry $d(\M_1,\M_2) = d(\M_1',\M_2')$ and consequently $d(\ts{(T_1)},\ts{(T_2)})$ is well-defined simply by fixing it to equal $d(\M_1,\M_2)$ for any choice of $\M_1$ and $\M_2$ generating the given tile sets $\ts{(T_1)},\ts{(T_2)}$.
Symmetry and the triangle inequality for $d$ when regarded as a function of pairs of tile sets follow immediately from the corresponding inequalities for matrices, and $d(\ts{(T_1)},\ts{(T_2)}) = 0$ if and only if $\max \{|\M_1^{-1} \M_2|,|\M_2^{-1} \M_1|\} \leq 1$ for suitable representative $\M_1$ and $\M_2$, which as seen before implies that $|\M_1^{-1} \M_2| = |\M_2^{-1} \M_1| = 1$, meaning that $\ts{(T_1)} = \ts{(T_2)}$ (also by Proposition \ref{pretileprop}).
\end{proof}

We come now to the main task of this section, which is to count the overlap multiplicity of dilated tiles $t T$ for $T \in {\mathcal T}_*$. When tiles are the similar shapes, there is an automatic algebraic bound on the overlap which comes directly from the fact that tiles have bounded multiplicity when not dilated. This is made precise in Proposition \ref{overlapcountprop}. Proposition \ref{overlapcount} then extends the counting result to the full generality of admissible tilings.
\begin{proposition}[Overlap of dilated tiles of similar shape] \label{overlapcountprop}
Suppose that $\mathcal{T}_*$ is an admissible tiling of $\R^d$ and that $\mathcal{T}_*' \subset \mathcal{T}_{*}$ is a subset for which there exists some metric ball $B_r$ of radius $r$ inside the space of origin-centered tile sets such that $\ts{\dot T}$ belongs to $B_r$ for all $T \in \mathcal{T}_*'$. Then for any $t \geq 1$ and any $\xi \in \R^d$, the subset of $\mathcal{T}_*'$ consisting of tiles $T$ with the property that $\xi  \in \ts{(t T)}$ is finite and has cardinality not exceeding $\const (2^{2r+1} t)^d$.
\end{proposition}
\begin{proof}
Fix $\xi$ and $T_1,\ldots,T_N \in \mathcal{T}_*'$ such that $\xi \in \ts{(t T_i)}$ for each $i$. Let each $T_i$ have the form $(\M_i,\xi_i)$ for appropriate choices of $\M_i$ and $\xi_i$ and let $T_0 := (\M_0,\xi)$ be such that $\ts{(\dot T_0)}$ is the center of $B_r$ (which may not belong to ${\mathcal{T}}_*$). By definition of the metric, $d(\M_i,\M_0) < r$ for each $i=1,\ldots,N$.

The proof is essentially geometric.
Suppose that $\eta$ is any point in the tile set $\ts{(t T_i)}$. This implies  $t^{-1} |\M_i^{-1} (\eta - \xi_i)| \leq 1/2$. Because $\xi$ also belongs to $\ts{(t T_i)}$, the triangle inequality implies that $t^{-1} |\M_i^{-1}(\eta - \xi)| \leq 1$.  By \eqref{comparability}, this means that $t^{-1} |\M_0^{-1}(\eta - \xi)| \leq 2^{d(\M_0,\M_i)} < 2^r$. In particular, this means that every tile set $\ts{(t T_i)}$ is contained in the tile set $\ts{(2^{r+1} t T_0)}$, which happens to have volume less than $(2^{r+1} t)^d |\det \M_0|$. Since $t \geq 1$, the union $\bigcup_{i=1}^N \ts{(T_i)}$ is also contained this single large set tile centered at $\xi$. This means that
\begin{equation}
\begin{split}
 \int {\const}^{-1} \sum_{i=1}^N \one_{T_i} (\xi') d \xi' & \leq  \int \indicator_{\bigcup_{i=1}^N \ts{(T_i)}}(\xi') d \xi' \\ & \leq \int \sum_{i=1}^N \one_{T_i} (\xi') d \xi' \leq \const \int \indicator_{2^{r+1} t T_0}(\xi') d \xi' 
 \end{split} \label{comparison}
 \end{equation}
because the sum of indicator functions $\one_{T_i}(\xi')$ is assumed to be less than or equal to $\const$ for almost every $\xi'$ and because the tile sets $\ts{(T_i)}$ are all contained in the tile set of $2^r t T_0$. But now the inequality $|\M_i^{-1} x| \leq 2^{r} |\M_0^{-1} x|$ for each $x$ implies that the tile set $\ts{(\dot T_i)}$ must contain a parallelepiped with the same volume as $\ts{(2^{-r} T_0)}$ but centered at $\xi_i$. As a consequence, $|\ts{(T_i)}| \geq 2^{-rd}|\det \M_0|$, where $|\cdot|$ on the left-hand side indicates Lebesgue measure. Therefore
\[  N \const^{-1} 2^{-rd} |\det \M_0| \leq  \left| \bigcup_{i=1}^N \ts{(T_i)} \right| \leq (2^{r+1} t)^d |\det \M_0|, \]
and hence it follows that $N \leq \const \left( 2^{2r + 1} t \right)^d$
as asserted.
\end{proof}

\begin{proposition}[Overlap of dilated tiles, any shape] \label{overlapcount}
Suppose $\mathcal{T}_*$ is an admissible tiling of $\R^d$. There exists a constant $\const'< \infty$ depending only on $\const$ and $d$ such that the following holds.
For any real $t \geq 1$ and $\xi \in \R^d$, let $\mathcal{O}_\xi(t)$ be the set of tiles $T \in \mathcal{T}_*$ such that $\xi \in \ts{(t T)}$. Then for almost every $\xi \in \R^d$, $\mathcal{O}_\xi(t)$ has cardinality not exceeding $\const' t^d K(t)$ for each $t \geq 1$.
\end{proposition}
\begin{proof}
Because $\mathcal{T}_*$ is admissible, almost every point $\xi$ necessarily belongs to $\ts{(T_\xi)}$ for some $T_\xi \in \mathcal{T}_*$ (though note that $\xi$ is not necessarily the \textit{center} of the tile set but merely belongs to it).  Divide the set $\mathcal O_\xi(t)$ into two parts $\mathcal{O}_{\xi}^{\mathrm{I}}(t)$ and $\mathcal{O}_{\xi}^{\mathrm{II}}(t)$: the subset of those $T$ for which $\ts{(\rho(T) T)} \cap \ts{(T_\xi)}$ is nonempty and those $T$ for which $\ts{(\rho(T) T)} \cap \ts{(T_\xi)}$ is empty, respectively.
For $T \in \mathcal{O}_{\xi}^{\mathrm{I}}(t)$, the regularity conditions on $\mathcal{T}_*$ imply that $d(\dot T, \dot T_\xi) \leq \const$. By Proposition \ref{overlapcountprop}, the cardinality of $\mathcal{O}_{\xi}^{\mathrm{I}}(t)$ at most $\const (2^{2\const+1} t)^d$.

For $T \in \mathcal{O}_{\xi}^{\mathrm{II}}(t)$, $\ts{(\rho(T) T)} \cap \ts{(T_\xi)}$ is empty, meaning in particular that it cannot contain $\xi$. Since $\ts{(t T)} \cap \ts{(T_\xi)}$ is \textit{not} empty (it contains $\xi$), it follows that $t > \rho(T)$ for all  $T \in \mathcal{O}_{\xi}^{\mathrm{II}}(t)$. By admissibility of the tiling, this means that there exist finitely many tiles $T_1,\ldots,T_K$ centered at the origin with $K \leq K(t)$ such that every $T \in \mathcal{O}_{\xi}^{\mathrm{II}}(t)$ has $d(\dot T, T_i) < \const$ for some $i \in \{1,\ldots,K\}$. Applying Proposition \ref{overlapcountprop} to the subset of $\mathcal{O}_{\xi}^{\mathrm{II}}(t)$ whose translated tiles are close to $T_i$ and summing over $i$ gives that $\# \mathcal{O}_{\xi}^{\mathrm{II}}(t) \leq \const (2^{2\const+1} t)^d K(t)$. Summing the estimates for $\# \mathcal{O}_{\xi}^{\mathrm{I}}(t)$ and $\mathcal{O}_{\xi}^{\mathrm{II}}(t)$ and using the fact that $K(t) \geq 1$ gives the proposition with $\const' := 2 \const 2^{d(2\const+1)}$.
\end{proof}

\section{Proof of Theorem \ref{maintheorem}}
\label{mainproofsec}
\begin{proof}[Proof of Theorem \ref{maintheorem}]
Theorem \ref{maintheorem} will be established as a consequence of Theorem \ref{compactthm} by taking $g = \eta$ for a suitable choice of $\eta$ given by Corollary \ref{almostbl}. Just as in the proof of Theorem \ref{compactthm}, one has that
\begin{equation}
\sum_{T \in \tilefam} |\widehat{\eta}^T(\xi')|^2 \one_{(1-\theta) T}(\xi') \geq  \left( 1 - \frac{C'}{W(\epsilon^{-1})} \right)^{2d} \sum_{T \in \tilefam} \one_{(1-\theta)T}(\xi'), \label{lowerb}
\end{equation}
for every $\xi' \in \R^d$, and the sum on the right-hand side is always at least one almost everywhere because $\tilefam$ is admissible.
As for the upper frame bound, it is useful to divide the sum into terms $T$ for which $\xi' \in \ts{T}$ and those for which $\xi' \not \in \ts{T}$. Then
\begin{equation}
\sum_{T \in \tilefam} |\widehat{\eta}^T(\xi')|^2  \leq \sum_{T \in {\mathcal T}_*} \one_{T}(\xi') + \sum_{\xi' \not \in \ts{T}} |\widehat{\eta}^{T}(\xi')|^2.  \label{split1}
\end{equation}
Admissibility of $\mathcal{T}_*$ means that the first sum on the right-hand side is at most $\const$ for almost every $\xi'$. As for the second sum, suppose that $T = (D,\xi)$ is a tile such that
$\xi' \not \in \ts{T}$. This means that $|{D}^{-1}(\xi' - \xi)| > \frac{1}{2}$. In this regime, the ``tail'' inequality \eqref{schwartztail} holds for $\eta$ and consequently
\[ |\widehat{\eta}^{T}(\xi')|^2 \leq \frac{(C')^2}{(W(\epsilon^{-1} |{D}^{-1}(\xi' - \xi)|))^2}. \]
This subsequently implies that
\begin{align*}
 |\widehat{\eta}^{T}(\xi')|^2 & \leq \int_{2 |{D}^{-1} (\xi' - \xi)|}^{4 |{D}^{-1} (\xi' - \xi)|} \frac{(C')^2}{(W(\epsilon^{-1} |{D}^{-1}(\xi' - \xi)|))^2} dt \\
 & \leq   \int_{1}^{\infty} \frac{(C')^2 \one_{t T}(\xi')}{(W(\epsilon^{-1} t /4))^2} dt
 \end{align*}
 since $\one_{t T}(\xi')$ is one when $t \geq 2 |{D}^{-1}(\xi' - \xi)| \geq 1$, the interval from $2 |{D}^{-1} (\xi' - \xi)|$ to $4 |{D}^{-1} (\xi' - \xi)|$ has length at least one, and since $W(\epsilon^{-1} t/4) \leq W(\epsilon^{-1} |{D}^{-1}(\xi' - \xi)|)$ when $t \leq 4 |{D}^{-1}(\xi' - \xi)|$. Inserting this inequality into the right-hand side of \eqref{split1} and interchanging summation and integration gives that almost every $\xi' \in \R^d$ satisfies
 \begin{align}
  \sum_{T \in \tilefam} |\widehat{\eta}^T(\xi')|^2 & \leq \mathcal{C} + \int_1^\infty \frac{(C')^2 \sum_{T \in \tilefam} \one_{t T}(\xi')}{(W( \epsilon^{-1} t / 4 \epsilon ))^2} dt \nonumber \\
  & = \mathcal{C} + \int_1^\infty \frac{(C')^2 \#\mathcal{O}_{\xi'}(t)}{(W( \epsilon^{-1} t /4))^2} dt. \label{backhere}
  \end{align}
Let $K(t)$ be the function from the definition of the admissible tiling $\tilefam$ and fix
\[ W(t) := \begin{cases} 1 & |t| \leq 1, \\ e^{\sqrt{|t|}-1} K(|t|) & |t| \geq 1.\end{cases} \]
Clearly this $W(t)$ is even and radially nondecreasing, has $W(t) = 1$ on $[-1,1]$, grows faster than any power of $t$ (because $K$ is bounded below) and has
\[ \int_0^\infty \frac{\ln W(t)}{t^2} dt = \int_1^\infty \frac{\sqrt{t} - 1 + \ln K(t)}{t^2} dt = 1 + \int_1^\infty \frac{\ln K(t)}{t^2} dt < \infty. \]
For this $W$ and any $\epsilon \in (0,\frac{1}{4})$, it follows that 
\[  \int_1^\infty \frac{(C')^2 \#\mathcal{O}_\xi(t)}{W(\epsilon^{-1} t/4) } dt \leq (C')^2 \int_1^\infty \frac{\mathcal{C}' t^d}{e^{\sqrt{t} - 1}} dt,\]
by virtue of the upper bound for $\#\mathcal{O}_\xi(t)$ given by Proposition \ref{overlapcount}. Applying this inequality to \eqref{backhere} (noting that $(W( \epsilon^{-1} t/4))^{-2} \leq (W(\epsilon^{-1}/4))^{-1} (W(\epsilon^{-1} t/4))^{-1}$ on the domain of the integral) gives that
\[ \sum_{T \in \tilefam} |\widehat{\eta}^T(\xi)|^2 \leq \mathcal{C} + C'' e^{- \frac{1}{2 \sqrt{\epsilon}} + 1 } \]
for some constant $C''$ depending only on $\mathcal C$, $K(t)$, and $d$. Combining this with \eqref{lowerb} gives that 
\[ \left( 1 - \frac{C'}{W(\epsilon^{-1})} \right)^{2d} \leq \sum_{T \in \tilefam} |\widehat{\eta}^T(\xi')|^2 \one_{(1-\theta)T}(\xi')  \leq \sum_{T \in \tilefam} |\widehat{\eta}^T(\xi')|^2 \leq \mathcal{C} + C'' e^{- \frac{1}{2 \sqrt{\epsilon}} + 1 } \]
for almost every $\xi' \in \R^d$. For a suitable choice of $\epsilon$ (depending on $\const,K,\theta$, and $d$), the quantity on the left can be taken to be arbitrarily close to $1$ and the quantity on the right can be taken as close as desired to $\const$. Applying Theorem \ref{compactthm} for its own suitably small $\epsilon$ gives \eqref{cframe}, which in this case corresponds to the frame condition \eqref{themainframe} promised by Theorem \ref{maintheorem}.
\end{proof}


\begin{bibdiv}
\begin{biblist}

\bib{MR1350699}{article}{
      author={Benedetto, John~J.},
      author={Heil, Christopher},
      author={Walnut, David~F.},
       title={Differentiation and the {B}alian-{L}ow theorem},
        date={1995},
        ISSN={1069-5869},
     journal={J. Fourier Anal. Appl.},
      volume={1},
      number={4},
       pages={355\ndash 402},
         url={https://doi.org/10.1007/s00041-001-4016-5},
      review={\MR{1350699}},
}

\bib{MR4044680}{article}{
      author={Bytchenkoff, Dimitri},
      author={Voigtlaender, Felix},
       title={Design and properties of wave packet smoothness spaces},
        date={2020},
        ISSN={0021-7824},
     journal={J. Math. Pures Appl. (9)},
      volume={133},
       pages={185\ndash 262},
         url={https://doi.org/10.1016/j.matpur.2019.05.006},
      review={\MR{4044680}},
}

\bib{MR1066587}{article}{
      author={Daubechies, Ingrid},
       title={The wavelet transform, time-frequency localization and signal
  analysis},
        date={1990},
        ISSN={0018-9448},
     journal={IEEE Trans. Inform. Theory},
      volume={36},
      number={5},
       pages={961\ndash 1005},
         url={https://doi.org/10.1109/18.57199},
      review={\MR{1066587}},
}

\bib{MR2362408}{article}{
      author={Demanet, Laurent},
      author={Ying, Lexing},
       title={Wave atoms and sparsity of oscillatory patterns},
        date={2007},
        ISSN={1063-5203},
     journal={Appl. Comput. Harmon. Anal.},
      volume={23},
      number={3},
       pages={368\ndash 387},
         url={https://doi.org/10.1016/j.acha.2007.03.003},
      review={\MR{2362408}},
}

\bib{MR2179584}{article}{
      author={Feichtinger, Hans~G.},
      author={Fornasier, Massimo},
       title={Flexible {G}abor-wavelet atomic decompositions for
  {$L^2$}-{S}obolev spaces},
        date={2006},
        ISSN={0373-3114},
     journal={Ann. Mat. Pura Appl. (4)},
      volume={185},
      number={1},
       pages={105\ndash 131},
         url={https://doi.org/10.1007/s10231-004-0130-8},
      review={\MR{2179584}},
}

\bib{MR4302195}{article}{
      author={Gressman, Philip~T.},
      author={Urheim, Ellen},
       title={Multilinear oscillatory integral operators and geometric
  stability},
        date={2021},
        ISSN={1050-6926},
     journal={J. Geom. Anal.},
      volume={31},
      number={9},
       pages={8710\ndash 8734},
         url={https://doi.org/10.1007/s12220-020-00383-5},
      review={\MR{4302195}},
}

\bib{MR3010124}{article}{
      author={Gr\"{o}chenig, Karlheinz},
      author={Malinnikova, Eugenia},
       title={Phase space localization of {R}iesz bases for
  {$L^2(\Bbb{R}^d)$}},
        date={2013},
        ISSN={0213-2230},
     journal={Rev. Mat. Iberoam.},
      volume={29},
      number={1},
       pages={115\ndash 134},
         url={https://doi.org/10.4171/RMI/715},
      review={\MR{3010124}},
}

\bib{guth}{misc}{
      author={Guth, Larry},
       title={Decoupling estimates in {F}ourier analysis},
   publisher={arXiv:2207.00652},
        date={2022},
         url={https://arxiv.org/abs/2207.00652},
}

\bib{MR1916862}{article}{
      author={Hern\'{a}ndez, Eugenio},
      author={Labate, Demetrio},
      author={Weiss, Guido},
       title={A unified characterization of reproducing systems generated by a
  finite family. {II}},
        date={2002},
        ISSN={1050-6926},
     journal={J. Geom. Anal.},
      volume={12},
      number={4},
       pages={615\ndash 662},
         url={https://doi.org/10.1007/BF02930656},
      review={\MR{1916862}},
}

\bib{MR1574706}{article}{
      author={Ingham, A.~E.},
       title={A {N}ote on {F}ourier {T}ransforms},
        date={1934},
        ISSN={0024-6107},
     journal={J. London Math. Soc.},
      volume={9},
      number={1},
       pages={29\ndash 32},
         url={https://doi.org/10.1112/jlms/s1-9.1.29},
      review={\MR{1574706}},
}

\bib{MR2066831}{incollection}{
      author={Labate, Demetrio},
      author={Weiss, Guido},
      author={Wilson, Edward},
       title={An approach to the study of wave packet systems},
        date={2004},
   booktitle={Wavelets, frames and operator theory},
      series={Contemp. Math.},
      volume={345},
   publisher={Amer. Math. Soc., Providence, RI},
       pages={215\ndash 235},
         url={https://doi.org/10.1090/conm/345/06250},
      review={\MR{2066831}},
}

\bib{MR3189276}{article}{
      author={Nielsen, Morten},
       title={Frames for decomposition spaces generated by a single function},
        date={2014},
        ISSN={0010-0757},
     journal={Collect. Math.},
      volume={65},
      number={2},
       pages={183\ndash 201},
         url={https://doi.org/10.1007/s13348-013-0091-6},
      review={\MR{3189276}},
}

\bib{MR2885560}{article}{
      author={Nielsen, Morten},
      author={Rasmussen, Kenneth~N.},
       title={Compactly supported frames for decomposition spaces},
        date={2012},
        ISSN={1069-5869},
     journal={J. Fourier Anal. Appl.},
      volume={18},
      number={1},
       pages={87\ndash 117},
         url={https://doi.org/10.1007/s00041-011-9190-5},
      review={\MR{2885560}},
}

\end{biblist}
\end{bibdiv}

\end{document}